\documentclass{article}
\usepackage{amsmath,amssymb,amsthm,stmaryrd,makeidx}
\usepackage[all]{xy}

\DeclareMathOperator{\enm}{End}
\DeclareMathOperator{\ext}{Ext}

\DeclareMathOperator{\id}{id}

\DeclareMathOperator{\matr}{M}

\newcommand{\grmcat}[1]

\newcommand{\defm}[1]{\mathsf{Def}_{#1}}

\input xypic

\newtheorem{proposition}{Proposition}
\newtheorem{theorem}{Theorem}
\newtheorem{lemma}{Lemma}
\newtheorem{corollary}{Corollary}
\newtheorem{definition}{Definition}

\makeindex

\begin{document}
\author{Arvid Siqveland}
\title{Associative Local Function Rings}

\maketitle

\begin{abstract} 
We prove that for an arbitrary field  $k,$ a complete, associative $k^r$-algebra $\hat H$ augmented over $k^r$ has exactly $r$ maximal two-sided ideals and deserves the name $r$-pointed. If $A$ is any $k$-algebra, $M=\{M_i\}_{i=1}^r$ is a family of simple right $A$-modules with a countable $k$-basis, and there is a homomorphism
$\rho_A:A\rightarrow\enm_{\hat H}(H\hat{\otimes}_{k^r}(\oplus_{i=1}^r M_i))=:\hat O(M)$ then $\hat O(M)$ is $r$-pointed and $M$ is contained in the set of right simple  $\hat O(M)$-modules. Our main result is that the subalgebra generated $\rho_A(A)$ and all $\rho_A(a)^{-1}$ whenever $\rho_A(a)$ is a unit, is a natural substitute for the localization $A(M)$ of the $k$-algebra $A$ in $M$ which only exists when the Ore condition is fulfilled.
\end{abstract}

\section{Preliminaries}
We let $k$ denote an arbitrary field. All $k$-algebras are associative with unit and $k$ in  the center. All modules will be right modules at default, and for an $A$-module $M,$ $\eta^A_M:A\rightarrow\enm_k(M)$ denotes the structure morphism, that is $\eta^A_M(a)(m)=ma.$ 
We will only consider modules over $k$-algebras, and the modules are supposed to have a countable basis.
When we use the word \emph{ideal}, we mean two-sided ideal.

The following is well known, and can be found in e.g. \cite{Lam01}:

\begin{lemma}\label{localeqlemma} The following is equivalent for a unital ring $A$:

(i) $A$ has a unique maximal left ideal.

(ii) $A$ has a unique maximal right ideal.

(iii) For every $x\in A,$ either $x$ or $1-x$ is a (two sided) unit.
\end{lemma}

In the above case, the unique left and right ideals coincide with the Jacobson radical.

\begin{definition} A ring which satisfies one of the equivalent conditions in Lemma \ref{localeqlemma} is called a local ring.
\end{definition}

\begin{lemma}\label{uniquemaxlemma} If $A$ is a local ring, then $A$ has a unique maximal ideal $\mathfrak m.$
\end{lemma}

\begin{proof} This follows because the unique maximal right ideal coincides with the Jacobson radical which is also a two-sided ideal. 
\end{proof}

Notice that the converse to Lemma \ref{uniquemaxlemma} does not hold. It is relatively easy to prove that $\matr_k(2)$ has $(0)$ as the unique maximal ideal: Assume an ideal $\mathfrak a$ contains a  $2\times 2$-matrix $A\neq 0.$ By left multiplication (row reduction) and right multiplication (column reduction) we might assume $A=\left(\begin{matrix}0&1\\0&0\end{matrix}\right).$ Then $\left(\begin{matrix}0&1\\1&0\end{matrix}\right)\left(\begin{matrix}0&1\\0&0\end{matrix}\right)=\left(\begin{matrix}0&0\\0&1\end{matrix}\right)$ is in $\mathfrak a,$ and $\left(\begin{matrix}0&1\\0&0\end{matrix}\right)\left(\begin{matrix}0&1\\1&0\end{matrix}\right)=\left(\begin{matrix}1&0\\0&0\end{matrix}\right)$ is in $\mathfrak a,$ and so their sum $1\in\mathfrak a.$ On the other hand, it is not true that for every matrix $A,$ $A$ or $1-A$ 	is a unit.

\begin{lemma} Let $k$ be a field and $V$ a $k$-vector space. If $V$ is finite-dimensional, $\enm_k(V)$ has no non-trivial ideals. If $V$ has a countable basis, then the set $\mathfrak m\subseteq\enm_k(V)$ consisting of endomorphisms with finite dimensional image, is a unique maximal ideal.
\end{lemma}

\begin{proof} Lam, exercise 36, p. 50 \cite{Lam01}. 
\end{proof}

\begin{definition} A left (right) module $M$ over an associative ring $A$ is simple if it has no nonzero left (right) proper submodules. 
\end{definition}

\begin{lemma}\label{forwardlemma} Let $f:A\rightarrow B$ be a homomorphism of two associative algebras.  If $M$ is a  right $B$-module which is simple as right $A$-module, then $M$ is simple as $B$-module. 
\end{lemma}

\begin{proof} Every $B$-submodule $N\subseteq M$ is also an $A$-submodule, thus either $0$ or $M.$ 
\end{proof}

If $A$ is a local ring it has a unique maximal two-sided ideal, but the converse is not true. It is well known that for any ring $R,$ the ring of $n\times n$-matrices $\matr_n(R)$ are Morita equivalent to $R$ so that for a field $k,$ $\matr_n(k)$ has exactly one simple right module. Thus $\matr_n(k)$ has a unique maximal ideal, but it is not local. The unique simple right module can be defined by several non-isomorphic ideals.

\begin{definition} An associative ring $R$ with exactly $r$ maximal ideals, is called an $r$-pointed ring. When $r=1$ we call it a pointed ring.
\end{definition}

\section{Augmented $k$-algebras}
Let $k$ be a field and $r\in\mathbb N_+.$ Let $\{Q_i\}_{i=1}^r$ be a family of pointed $k$-algebras containing $k.$ A $k$-algebra $H$ is augmented over $\oplus_{i=1}^r Q_i$ if it  fits in the diagram 
$$\xymatrix{k^r\ar[r]^\rho\ar[dr]_{\iota}&H\ar@{>>}[d]^{\pi_H}\ar[dr]^{\tilde{\pi}_H}&\\&\oplus_{i=1}^r Q_i\ar@{>>}[r]&\oplus_{i=1}^r Q_i/\mathfrak m_i}$$ where $\mathfrak m_i\subset Q_i$ is the unique maximal ideal.

Let $\mathfrak m=\ker\tilde\pi_H$ which is a two-sided ideal, not necessarily maximal because $k^r\subseteq H/\mathfrak m$ is not a division ring for $r\geq 2.$

\begin{definition} The $\mathfrak m$-adic completion of $H$ is defined as $$\hat H=\underset{\underset{n\geq 0}\leftarrow}\lim H/\mathfrak m^n.$$ Then $\hat H$ fits in the diagram $$\xymatrix{k^r\ar[r]^\rho\ar[dr]_{\id}&\hat H\ar@{>>}[d]^\pi\\&\oplus_{i=1}^r Q_i,}$$ and by definition $\hat{\hat H}\simeq\hat H.$ Any $k$-algebra $S$ augmented over some $\oplus_{i=1}^r Q_i$ with the property that $\hat S\simeq S$ is called complete.
\end{definition}

\begin{lemma}\label{unitlemma}
Let $\hat H$ be a complete algebra augmented over $k^r.$ Let $\pi_i$ be the composition of the augmentation with the $i$'th projection, and let $\mathfrak m_i=\ker\pi_i.$ If $x\notin\cup_{i=1}^r\mathfrak m_i,$ then $x$ is a unit.
\end{lemma}

\begin{proof} We have the diagram \begin{equation}\label{rlocdiag}\xymatrix{k^r\ar[r]^\rho\ar[dr]_\id&\hat H\ar@{>>}[d]^\pi\\&k^r}\end{equation}
and $\pi(x)=\pi(\rho(\alpha))$ where $\alpha=(\alpha_1,\dots,\alpha_r)$ with all $\alpha_i\neq 0.$

Then $\rho(\alpha)-x=n\in\mathfrak m$ so that 
$x=\rho(\alpha)-n=\rho(\alpha)(1-\rho(\alpha^{-1})n)=\rho(\alpha)(1-m_1)$ and
$x=\rho(\alpha)-n=(1-n\rho(\alpha^{-1}))\rho(\alpha)=(1-m_2)\rho(\alpha).$

We find that 

$$\begin{aligned} x(\sum_{i=0}^\infty m_1^i)\rho(\alpha^{-1})=\rho(\alpha)(1-m)(\sum_{i=0}^\infty m_1^i)\rho(\alpha^{-1})=\rho(\alpha)\rho(\alpha^{-1})=1,\\
\rho(\alpha^{-1})(\sum_{i=0}^\infty m_2^i)x=\rho(\alpha^{-1})(\sum_{i=0}^\infty m_2^i)(1-m_2)\rho(\alpha)=\rho(\alpha^{-1})\rho(\alpha)=1.
\end{aligned}$$

This proves that $x$ is a unit, and the right and left unit are equal:
$u_1 x=1$ and $x u_2=1$ gives $u_1=u_1(x u_2)=(u_1x)u_2=u_2.$

\end{proof}

\begin{proposition}\label{rpointprop} A complete $k$-algebra $\hat H$ augmented over $k^r$ has exactly $r$ maximal right ideals. These are also left maximal ideals.
\end{proposition}

\begin{proof} Consider the diagram \ref{rlocdiag}. The ideals $\mathfrak m_i$ are clearly two-sided and maximal because $\hat H/\mathfrak m_i\simeq k.$ By Zorn's lemma, any proper right ideal is contained in a maximal right ideal. For any right ideal $\mathfrak a$ which is not contained in one of the $\mathfrak m_i$ we can choose $x_i\in\mathfrak a$ such that $x_i\notin\mathfrak m_i.$ Then $x=\sum_{i=1}^rx_ie_i\in\mathfrak a$ and $x\notin\cup_{i=1}^r\mathfrak m_i.$ Such an $x$ is a right unit by Lemma \ref{unitlemma} and the result follows.
\end{proof}

\begin{corollary} A complete $k$-algebra $\hat H$ augmented over $k$ is local.
\end{corollary}

\begin{proof} Follows from Lemma \ref{unitlemma} with $r=1.$
\end{proof}

Let $V_{ij},\ 1\leq i,j\leq r$ be $k$-vector spaces. Then the $r\times r$-matrix $V=(V_{ij})$ is a $k^r$-algebra, and we can consider the tensor algebra $F(V)=\oplus_{i\geq 0}V^{{\otimes_{k^r}}^i}.$ If we choose  bases $\{t_{ij}(l)\}_{l=1}^{l_{ij}}\subset V_{ij},$ we will write $$F(V)=F((l_{ij})_{r\times r})=\begin{tiny}\left(\begin{matrix}k\langle t_{11}(1),\dots,t_{11}(l_{11})\rangle&\cdots&\langle t_{1r}(1),\dots,t_{1r}(l_{1r})\rangle\\\dots&\langle t_{ii}(1),\dots,t_{ii}(l_{ii})\rangle&\dots\\
\langle t_{r1}(1),\dots,t_{r1}(l_{r1})\rangle&\cdots&k\langle t_{rr}(1),\dots,t_{rr}(l_{rr})\rangle
\end{matrix}\right).\end{tiny}$$
Then $F(V)$ is a $k$-algebra, augmented over $k^r,$ and in the special case $V_{ij}\simeq k,\ 1\leq i,j\leq r,$ we call $T=F(V)/(\ker\pi)^2$ the $r\times r$ tangent algebra.

\section{Completions of Algebras}

Let $\phi:R\rightarrow S$ be a morphism between two $r$-pointed $k^r$-algebras, and let $M$ be a $k^r$-module. Then we have an algebra homomorphism $\Pi(\phi):\enm_{R}(R\otimes_{k^r}M)\rightarrow\enm_S(S\otimes_{k^r}M)$ given by the identity $S\otimes_R(R\otimes_{k^r}M)\simeq S\otimes_{k^r}M.$

\begin{definition}\label{precompdef} Let $A$ be a $k$-algebra and $M=\oplus_{i=1}^rM_i$ a sum of right $A$-modules, $r\geq 1.$ A complete $k$-algebra $\hat H^A_M$ augmented over $k^r$ is a pre-completion of $A$ in $M$ if the following (versal) condition holds: There exists a homomorphism $\hat\rho^A_M:A\rightarrow\enm_{\hat H^A_M}(\hat H^A_M\otimes_{k^r}M)$ fitting in the diagram $$\xymatrix{A\ar[d]_\phi\ar[r]^-{\hat\rho^A_M}\ar[dr]_(0,3){\oplus\eta_{M_i}}&\enm_{\hat H^A_M}(\hat H^A_M\otimes_{k^r}M)\ar[d]^{\Pi(\pi_{\hat H})}\ar@{-->}[dl]_{\Pi(\psi)}\\\enm_S(S\otimes_{k^r}M)\ar[r]_{\Pi(\pi_S)}&\oplus_{i=1}^r\enm_k(M_i)}$$ and such that if $S$ is any other such algebra with a homomorphism $$\phi:A\rightarrow\enm_{S}(S\otimes_{k^r}M)$$ fitting in the diagram, then there exists a morphism $\psi:\hat H^A_M\rightarrow S,$ unique in the case  $S\cong T,$ commuting in the above diagram.
\end{definition}

In the book \cite{ELS17} we proved the following theorem.

\begin{theorem}\label{deftheorem} Let $A$ be a finitely generated associative $k$-algebra. Let $M=\oplus_{i=1}^rM_i$ be a sum of right $A$-modules such that for each $1\leq i,j\leq r$ we have that $\dim_k\ext^1_A(M_i,M_j)<\infty.$ Then there exists a prorepresenting hull for the deformation functor $\defm M.$
\end{theorem}

From this, it follows that a pre-completion as in Definition \ref{precompdef} exists under the given conditions, and that it is unique up to (non-unique) isomorphism. 

\begin{proposition} Assume that $A$ is a $k$-algebra augmented over $k^r$ and let $M_i=A/\ker(p_i\circ\pi_A)\simeq k,$ where $p_i:k^r\rightarrow k$ is the $i$th projection. Then the $\mathfrak m$-adic completion $\hat A$ of $A$ is a pre-completion of $A$ in $M=\oplus_{i=1}^r M_i.$
\end{proposition}

\begin{proof} In this situation, $M\simeq k^r$ such that $\enm_{\hat H^A_M}(\hat H^A_M\otimes M)\simeq\enm_{\hat H^A_M}(\hat H^A_M)\simeq\hat H^A_M$ and then the result follows by the universal property of projective limits.
\end{proof}

\begin{definition}
Let $A$ be a $k$-algebra and  $M=\oplus_{i=1}^rM_i$ a  sum of $r$ right $A$-modules. If $A$ has a pre-completion $\hat H^A_M$ we define the completion of $A$ in $M$ as $$\hat O^A_M=\enm_{\hat H^A_M}(\hat H^A_M\otimes_{k^r}M).$$
\end{definition}

By definition, the completion of $A$ in $M,$ if it exists, comes with a homomorphism $\rho^A_M:A\rightarrow\hat O^A_M$ commuting in the diagram $$\xymatrix{A\ar[r]^{\rho^A_M}\ar[dr]_{\oplus_{i=1}^r\eta_{M_i}}&\hat O^A_M\ar@{>>}[d]^\pi\\&\oplus_{i=1}^r\enm_k(M_i).}$$ 
Also, $\hat O^A_M$ is augmented over $\oplus_{i=1}^r\enm_k(M_i)$ where each $\enm_k(M_i)$ is a pointed $k$-algebra.

\begin{proposition} The completion of $A$ in $M$ if it exists, is complete.
\end{proposition}

\begin{proof} This is because $$\hat O^A_M=\enm_{\hat H^A_M}(\hat H^A_M\otimes_{k^r}M)\simeq\underset{\underset{n\geq 0}\leftarrow}\lim \enm_{\hat H^A_M}(\hat H^A_M\otimes_{k^r}M)/\mathfrak m^n.$$
\end{proof}

\begin{proposition} The completion $\hat O^A_M$ of $A$ in a sum of right  modules $M=\oplus_{i=1}^rM_i$ has $\tilde M=\{M_i\}_{i=1}^r$ as its set of simple right and left modules.
\end{proposition}

\begin{proof} It follows from Lemma \ref{forwardlemma} that $M$ is contained in the set of simple right modules. We also have the diagram
$$\xymatrix{k^r\ar[r]^\rho\ar[dr]_\iota&\hat A(M)\ar@{>>}[d]^\pi\\&\oplus_{i=1}^r\enm_k(M_i)/\mathfrak m_i}$$ where $\mathfrak m_i\subset\enm_k(M_i)$ is the unique maximal ideal. It then follows that $\hat O^A_M$ has exactly $r$ two-sided maximal ideals, and then the result follows by Morita-equivalence.
\end{proof}

\begin{proposition} Let $A$ be a commutative $k$-algebra, $k$ algebraically closed and $A$ finitely generated. Let $M=\oplus_{i=1}^r A/\mathfrak m_i$ where $\mathfrak m_i,\ 1\leq i\leq r,$ are maximal ideals. Then $$\hat O^A_M\simeq\oplus_{i=1}^r\hat A_{\mathfrak m_i}.$$ 
\end{proposition}

\begin{proof}
In this situation we know that for each $1\leq i\leq r$ that the $\mathfrak m_i$-adic completion $\hat A_{\mathfrak m_i}$ of $A$ is a prorepresenting hull for $\defm M.$ The rest follows from the fact that $A/\mathfrak m_i\simeq k,$ so that $\enm_{\hat H^A_M}(\hat H^A_M\otimes_{k^r}M)\simeq \hat H^A_M.$ 
\end{proof}

\section{The Ring of Local Functions}

In algebraic and differential geometry, the rings of local functions in  closed points are sufficient for the classification by geometric properties. If we want the possibility to parametrize finite sets of interacting points, we will have to evaluate functions in finite sets of closed points.

From Theorem \ref{deftheorem} it follows that when $A$ is a $k$-algebra and $M=\oplus_{i=1}^r M_i$ is a sum of right $A$-modules that satisfies the given conditions, the completion $\hat O^A_M$ exists. It is important to notice that the deformation theory is just a proof of existence, and that the computation of pre-completion is explicitly and algorithmically given, independent of deformation theory in \cite{S242}.

\begin{definition} Let $M=\oplus_{i=1}^r M_i$ be a sum of right $A$-modules. Then the ring $O^A_M$ of functions locally defined at $M$ is defined as the sub-algebra of $\hat O^A_M$ generated by  $\rho^A_M(A)$ (over $k^r$), together with $\rho^A_M(a)^{-1}$ whenever $\rho^A_M(a)$ is a unit.
\end{definition}

\begin{proposition} The $k$-algebra $O^A_M$ is augmented over $\oplus_{i=1}^r\enm_k(M_i),$ and $$O^{O^A_M}_M\simeq O^A_M.$$
\end{proposition}

\begin{proof} This follows from the universal property of projective limits, and the versal property of prorepresenting hulls (or pre-completion). Then we have $\hat O^{O^A_M}_M\simeq\hat O^A_M,$ from which the result follows.
\end{proof}

Consider a ring homomorphism $\phi:A\rightarrow B$ between two associative rings. Then every $B$-module $M$ is also an $A$-module via $\phi.$ More precisely, if $M$ is defined as $B$-module by the structure morphism $\eta^B_M,$ then $M$ is defined as an $A$-module by the structure morphism $\eta^A_M=\eta^B_M\circ\phi.$ We follow the usual notation, and call the $B$ module $M$ considered as $A$-module via $\phi$, $M^c.$

\begin{lemma} Let $B$ be a $k$-algebra and $M=\oplus_{i=1}^rM_i$ a sum of right $B$-modules.
Let $\phi:A\rightarrow B$ be an algebra homomorphism so that $M^c=\oplus_{i=1}^r M^c_i.$ Then $\phi$ induces a morphism of augmented algebras $\phi^\ast:O^A_{M^c}\rightarrow O^B_M$.
\end{lemma}

\begin{proof} This follows directly from the versal property of pre-completion (or prorepresentable hull). The morphism $\phi:A\rightarrow B$ induces a homomorphism $A\rightarrow\hat O^A_{M^c}$ in the diagram
$$\xymatrix{A\ar[drr]_{\oplus_{i=1}^r\eta^A_{M^c}}\ar[r]^\phi&B\ar[dr]^{\oplus_{i=1}^r\eta^B_M}\ar[r]^-{\hat\rho^B_M}&\hat O^B_M=\enm_{\hat H^B_M}(\hat H^B_M\otimes_{k^r}M)\ar[d]^{\Pi(\pi_{\hat H})}\\
&&\oplus_{i=1}^r\enm_k(M_i)}$$
and $O^B_M\subseteq\hat O^B_M$ is the subalgebra generated by $\rho^B_M(B)$ together with the set $\{\rho(b)^{-1}|\rho^B_M(b)\text{ is a unit}\}.$ From the definition of precompletion, there is a homomorphism $\psi:\hat H^A_{M_c}\rightarrow\hat B_M$ which is surjective because it is so when restricted to $\psi_2:\hat H^A_{M_c}/(\ker\pi_{\hat H^A_{M_c}})^2\rightarrow\hat H^B_{M}/(\ker\pi_{\hat H^B_{M}})^2.$ Also the homomorphism $\psi$ induces the diagram

$$\xymatrix{A\ar[d]_{\hat\rho^B_M\circ\phi}\ar[r]^-{\hat\rho^A_M}\ar[dr]_(0,3){\oplus\eta_{M_i}}&\enm_{\hat H^A_M}(\hat H^A_M\otimes_{k^r}M)\ar[d]^{\Pi(\pi_{\hat H^A_M})}\ar@{-->}[dl]_{\Pi(\psi)}\\\enm_{\hat H^B_M}(\hat H^B_M\otimes_{k^r}M)\ar[r]_{\Pi(\pi_{\hat H^B_M})}&\oplus_{i=1}^r\enm_k(M_i).}$$

 Then, if $\hat\rho^A_M(a)$ is invertible in $\hat O^A_M,$ $\hat\rho^B_M(\phi(a))=\Pi(\psi)(a)$ is invertible in $\hat O^B_M$ and this gives a morphism $\phi^\ast:O^A_{M_c}\rightarrow O^B_M.$
 
\end{proof}

\end{document}